\documentclass[12pt,leqno]{article}
\usepackage{amsmath,amssymb,array}

\usepackage{amsthm}

\usepackage[all,cmtip]{xy}

\usepackage{url}

\makeatletter
\def\@citex[#1]#2{\if@filesw\immediate\write\@auxout{\string\citation{#2}}\fi
  \def\@citea{}\@cite{\@for\@citeb:=#2\do
    {\@citea\def\@citea{\@citesep}\@ifundefined
       {b@\@citeb}{{\bf ?}\@warning
       {Citation `\@citeb' on page \thepage \space undefined}}%
{\csname b@\@citeb\endcsname}}}{#1}}
\def\@citesep{; }
\makeatother

\font\cyr=wncyr10 scaled \magstep1%

\def\Sha{\text{\cyr Sh}}

\def\nr{_{\textrm{nr}}}

\DeclareMathOperator{\Br}{Br}
\DeclareMathOperator{\Outc}{Out_{\text{\rm{c}}}}
\DeclareMathOperator{\Autc}{Aut_{\text{\rm{c}}}}

\newtheoremstyle{Kang}{}{}{\itshape}{}{\bf}{}{.5em}{}
\theoremstyle{Kang}
\newtheorem{theorem}{Theorem}[section]

\newtheorem{lemma}[theorem]{Lemma}
\newtheorem{coro}[theorem]{Corollary}

\newtheorem{prop}[theorem]{Proposition}

\newtheoremstyle{Kremark}{}{}{}{}{\bf}{}{.5em}{}
\theoremstyle{Kremark}
\newtheorem*{remark}{Remark.}
\newtheorem{defn}[theorem]{Definition}

\newtheorem{other}{}

\newtheorem{question}[theorem]{Question}

\allowdisplaybreaks[1]  

\title{The Bogomolov multiplier of rigid finite groups}
\author{Ming-chang Kang \\[3mm]
Department of Mathematics and \\ Taida Institute of Mathematical Sciences,\\
National Taiwan University \\ Taipei, Taiwan \\
E-mail: kang@math.ntu.edu.tw \\[5mm]
Boris Kunyavski\u\i  \\[3mm]
Department of Mathematics,\\
Bar-Ilan University \\ Ramat Gan, Israel \\
E-mail: kunyav@macs.biu.ac.il}

\date{}

\begin{document}

\maketitle

\footnote{\textit{\!\!\! Mathematics Subject Classification $(2010)$}: Primary 14E08, 14L30, 20D45, Secondary 20J06.}
\footnote{\textit{\!\!\! Keywords}:
Bogomolov multiplier, unramified Brauer group, Shafarevich--Tate set, class-preserving automorphisms.}

\begin{abstract}
{\noindent\bf Abstract.} The Bogomolov multiplier of a finite group $G$ is defined as the subgroup
of the Schur multiplier consisting of the cohomology classes vanishing after
restriction to all abelian subgroups of $G$.  This invariant of $G$ plays an
important role in birational geometry of quotient spaces $V/G$. We show that
in many cases the vanishing of the Bogomolov multiplier is guaranteed by
the rigidity of $G$ in the sense that it has no outer class-preserving automorphisms.
\end{abstract}

\section{Introduction}

The main object of this note is the following invariant of a finite
group $G$:

\begin{equation} \label{def:B0}
B_0(G)=\ker [H^2(G,\mathbb Q/\mathbb Z)\to \bigoplus_{A\subset
G}H^2(A,\mathbb Q/\mathbb Z)]
\end{equation}
where $A$ runs over all abelian subgroups of $G$. Bogomolov showed
in \cite{Bo} that this group coincides with the unramified Brauer
group $\Br\nr(V/G)$ where $V$ is a vector space defined over an
algebraically closed field $k$ of characteristic zero equipped
with a faithful, linear, generically free action of $G$. The
latter group is an important birational invariant of the quotient
variety $V/G$, introduced by Saltman in \cite{Sa}, \cite{Sa2}. He
used it for producing the first counter-example (for $G$ of order
$p^9$) to a problem by Emmy Noether on rationality of fields of
invariants $k(x_1,\dots ,x_n)^G$, where $k$ is algebraically
closed and $G$ acts on the variables $x_i$ by permutations.
Formula (\ref{def:B0}) provides a purely group-theoretic intrinsic
recipe for the computation of $\Br\nr(V/G)$. In the same paper
\cite{Bo} Bogomolov showed that it can be simplified even further:
one can replace $A$ with the set of all {\it bicyclic} subgroups
of $G$. In the case where $G$ is a $p$-group, he also suggested a
more explicit way for computing $B_0(G)$ and produced smaller
counter-examples. For some further activity concerning the values
of $B_0(G)$ for $p$-groups, as well as for corrigenda to some
assertions of \cite{Bo}, the interested reader is referred to
\cite{CHKK},  \cite{HKK},  \cite{Mo1}. In particular, it turned
out that the smallest power of $p$ for which there exists a
$p$-group $G$ with $B_0(G)\ne 0$ is 5 (for odd $p$) and not 6, as
claimed in \cite{Bo}; see \cite{Mo1} for details.

In the present paper, our viewpoint is a little different. Namely,
we address the following question: what group-theoretic properties
of $G$ can guarantee that $B_0(G)=0$? The first large family of
groups (outside $p$-groups) for which we have $B_0(G)=0$ is that of
all {\it simple} groups \cite{Ku1}. Thus, a natural question to ask
is what common properties, shared by simple groups and ``small''
$p$-groups, are responsible for vanishing of $B_0(G)$. Our vague
answer is that in a certain sense, both are {\it rigid}.

More precisely, the rigidity property we are talking about is the
following one. Let a group $G$ act on itself by conjugation, and let
$H^1(G,G)$ be the first cohomology pointed set. Denote by $\Sha (G)$
the subset of $H^1(G,G)$ consisting of the cohomology classes
becoming trivial after restricting to every cyclic subgroup of $G$
and call it the Shafarevich--Tate set of $G$ (this terminology was
introduced by T.~Ono \cite{On}, alluding to arithmetic-geometric
counterparts arising from the action of the Galois group of a number
field $k$ on the set of rational points of an algebraic $k$-group).
We say that $G$ is $\Sha$-rigid if the set $\Sha (G)$ consists of
one element; see \cite{Ku2} where this terminology was introduced in
view of relationships with other rigidity properties of $G$. In the
case where $G$ is finite, $\Sha (G)$ coincides with another
local-global invariant $\Outc (G)$, which was introduced by Burnside
\cite{Bu1} about a century ago: it is the quotient of the group
$\Autc(G)$ of class-preserving automorphisms of $G$ by the subgroup
of inner automorphisms (an automorphism is called class-preserving
if it moves each conjugacy class of $G$ to itself). In particular,
if $G$ is finite, then $\Sha (G)$ is a finite group, and $G$ is
$\Sha$-rigid if and only if every locally inner automorphism
$\varphi\colon G\to G$ (i.e., $\varphi (g)=aga^{-1}$ for some $a$
depending on $g$) is inner (i.e., $a$ can be chosen independent of
$g$).

Certain classes of finite groups are known to consist of
$\Sha$-rigid groups. The following proposition collects some data
from various sources.

\begin{prop} \label{prop:rigid}
The following finite groups are $\Sha$-rigid:

\begin{itemize}
\item[\rm{(i)}] symmetric groups {\rm{\cite{OW};}}
\item[\rm{(ii)}] simple groups {\rm{\cite{FS};}} \item[\rm{(iii)}]
$p$-groups of order at most $p^4$ {\rm{\cite{KV1};}}
\item[\rm{(iv)}] $p$-groups having a cyclic maximal subgroup
{\rm{\cite{KV2};}} \item[\rm{(v)}] $p$-groups having a cyclic
subgroup of index $p^2$ {\rm{\cite{KV3}, \cite{FN};}}
\item[\rm{(vi)}] abelian-by-cyclic groups {\rm{\cite{HJ};}}
\item[\rm{(vii)}] groups such that the Sylow $p$-subgroups are
cyclic for odd $p$, and either cyclic, or dihedral, or generalized
quaternion for $p=2$ {\rm{\cite{He1} (}}see {\rm{\cite{Su},
\cite{Wa}}} for a classification of such groups$);$
\item[\rm{(viii)}] Blackburn groups {\rm{\cite{He2}, \cite{HL};}}
\item[\rm{(ix)}] extraspecial $p$-groups {\rm{\cite{KV2};}}
\item[\rm{(x)}] primitive supersolvable groups {\rm{\cite{La};}}
\item[\rm{(xi)}] unitriangular matrix groups over $\mathbb F_p$
and the quotients of their lower central series {\rm{\cite{BVY};}}
\item[\rm{(xii)}] central products of $\Sha$-rigid groups
{\rm{\cite{KV2}.}}
\end{itemize}
\end{prop}

See \cite{Ya2} for a survey and some details.

Our main result states that the Bogomolov multiplier of most of the
groups listed above is trivial.

\begin{theorem} \label{th:main}
Let $G$ be one of the groups listed in items {\text{\rm{(i)--(ix)}}} of
Proposition $\ref{prop:rigid}$. Then $B_0(G)=0$.
\end{theorem}

This theorem is proved in \ref{sec:main} Some open questions arising
from this ``experimental'' observation are briefly discussed in
\ref{sec:concl}

\bigskip

{\noindent} {\it Notational conventions}.
Unless otherwise stated, $G$ denotes a finite group and $k$ stands
for an algebraically closed field of characteristic zero.

\section{Main results and proofs} \label{sec:main}

We start the proof of Theorem \ref{th:main} by observing that most
of the work had already been done earlier. Namely, the assertions
referring to items (i)--(vii) of Proposition \ref{prop:rigid} can
be extracted from the literature, sometimes in a somewhat stronger
form, stating that the relevant quotient varieties $V/G$ are
retract rational, stably rational, or even rational. Item (i) is a
direct consequence of the classical theorem by Emmy Noether
asserting the rationality of the field of invariants $k(x_1,\dots
,x_n)^{S_n}$ with respect to the natural permutation action of the
symmetric group $S_n$ (which follows from the theorem on
elementary symmetric functions). The rationality of $V/G$ is also
known in cases (iii) \cite{CK}, (iv) \cite{HK}, (v) \cite{Ka1}. In
case (vi) the variety $V/G$ is retract rational \cite{Ka2}, which
is weaker than rationality but enough to guarantee vanishing of
$B_0(G)$ \cite[Proposition 1.8]{Sa2}. The Bogomolov multiplier is
zero in case (ii) \cite{Ku1}. In case (vii), one can notice that
in view of \cite{Bo}, \cite{BMP}, it is enough to establish that
$B_0(S)=0$ for all Sylow subgroups $S$ of $G$. This is obvious for
odd primes because in that case $S$ is cyclic, and the groups
appearing in the case $p=2$ are all included in case (iv) above.
Thus it remains to consider cases (viii) and (ix), which
constitute the main body of the paper. They are treated separately
below.

\begin{prop} \label{prop:extra}
If $G$ is an extraspecial $p$-group, then $B_0(G)=0$.
\end{prop}

Before starting the proof, we present the following useful
observation. Recall that groups $G_1$ and $G_2$ are called {\it
isoclinic} if they have isomorphic quotients $G_i/Z(G_i)$ and
derived subgroups $[G_i,G_i]$, and these isomorphisms are compatible
(see, e.g., \cite[p.~285]{Be}).

\begin{lemma} \cite{Mo2} \label{lem:iso}
If $G_1$ and $G_2$ are isoclinic, then $B_0(G_1)\cong B_0(G_2)$.
\end{lemma}

\begin{remark} \label{rem:iso}
The assertion of this lemma was stated in \cite{HKK} as a
conjecture. It was generalized in \cite{BB} by showing that the
quotient varieties $V/G_i$ of isoclinic groups are stably
birationally equivalent. Note also a striking parallel with a result
of Yadav \cite{Ya1}, establishing the isomorphism $\Sha (G_1)\cong
\Sha (G_2)$ for isoclinic groups.
\end{remark}

\begin{proof}[Proof of Proposition $\ref{prop:extra}$]
Recall that the centre $Z$ of $G$ is of order $p$ and the quotient
$G/Z$ is a (nontrivial) elementary abelian $p$-group of order
$p^{2n}$. There is a classification of such groups (see, e.g.,
\cite[pp.~203--208]{Go}). Since all elementary $p$-groups of the
same order are isoclinic, in light of Lemma \ref{lem:iso} we may and
will consider only groups of exponent $p$. So from now on
\begin{equation} \label{eq:extra}
G=\left<z,x_1,\dots x_n,x_{n+1},\dots ,x_{2n} \, | \,
[x_i,x_{i+n}]=z, \, i=1,\dots ,n\right>
\end{equation}
(all other generators commute and are all of exponent $p$).

Our computations are based on \cite[Lemma~5.1]{Bo} (we use the
notation of \cite[Section~5]{Pe}). Namely, for a vector space
$E/\mathbb F_p$ we denote by $E^{\vee}$ the dual space. We identify
$\bigwedge^i(E^{\vee})$ with $(\bigwedge^iE)^{\vee}$ and denote it
by $\bigwedge^iE^{\vee}$. For any subset $B$ in $\bigwedge^iE$
(resp. $\bigwedge^iE^{\vee}$) we denote by $B^{\perp}$ its
orthogonal in $\bigwedge^iE^{\vee}$ (resp. $\bigwedge^iE$). We view
the abelian $p$-groups $Z=\left<z\right>$ and $G/Z=\left<\bar x_i,
i=1,\dots ,2n\right>$ as vector spaces over $\mathbb F_p$ and denote
them by $V$ and $U$ respectively (to ease the notation, we suppress
bars over the $x_i$ throughout below). Then we have the following
central extension of vector spaces
$$
0\to V\to G\to U\to 0,
$$
which gives rise to a surjective linear map
$\gamma\colon\bigwedge^2U\to V$ and the induced injective dual map
$\gamma^{\vee}\colon V^{\vee}\to\bigwedge^2U^{\vee}$. Let $K^2$
denote the image of $\gamma^{\vee}$, and let
$S^2=(K^2)^{\perp}\subset\bigwedge^2U$. Let $S^2_{\text{\rm{dec}}}$
be the subgroup of $S^2$ generated by the decomposable elements of
the form $u\wedge v$ ($u,v\in U$). Finally, let
$K^2_{\text{\rm{max}}}\supset K^2$ be the orthogonal to
$S^2_{\text{\rm{dec}}}$ in $\bigwedge^2U^{\vee}$. Then by
\cite[Lemma~5.1]{Bo} we have an isomorphism $B_0(G)\cong
K^2_{\text{\rm{max}}}/K^2$.

In our case, we have
$$
\gamma^{\vee}(\check{z})=\sum_{i=1}^n \check{x}_i \wedge
\check{x}_{i+n},
$$
where \v{} indicates to elements of the dual basis. Hence
$S^2\subset \bigwedge^2U$ is the hyperplane
$$
\{\sum_{i<j}\alpha_{i,j} x_i\wedge x_j \, | \, \sum_{i=1}^n
\alpha_{i,i+n}=0\}.
$$
It is spanned by the elements $x_i\wedge x_j$ ($j>i$, $j\ne i+n)$
and $x_i\wedge x_{i+n}-x_n\wedge x_{2n}$
($i=1,\dots ,n-1$). Each of the latter elements can be represented in the form
$$
x_i\wedge x_{i+n}-x_n\wedge x_{2n} = (x_i-x_n)\wedge
(x_{i+n}+x_{2n}) - x_i\wedge x_{2n} + x_n \wedge x_{i+n},
$$
i.e., as a sum of decomposable elements of $S^2$. Hence each of the
generators of $S^2$ belongs to $S^2_{\text{\rm{dec}}}$, and we have
$S^2=S^2_{\text{\rm{dec}}}$, whence $K^2_{\text{\rm{max}}}=K^2$, so
$B_0(G)=0$.
\end{proof}

This result can be extended to another class of groups, so-called
almost extraspecial groups. Recall (see, e.g., \cite{CT}) that a
$p$-group $G$ is called {\it almost extraspecial} if its centre $Z(G)$ is
cyclic of order $p^2$, and the Frattini subroup $\Phi (G)$ coincides
with the derived subgroup $[G,G]$ and they are both cyclic of order
$p$. Any such group is of order $p^{2n+2}$, $n\ge 1$, and any two
almost extraspecial groups of the same order are isomorphic.

\begin{coro}
If $G$ is an almost extraspecial $p$-group, then $B_0(G)=0$.
\end{coro}

\begin{proof}
The subgroup $H$ of $G$ generated by all elements of order $p$ is
extraspecial of order $p^{2n+1}$. If we denote by $z$ a generator of
$Z(G)$, then $z^p$ can be taken as a generator of $Z(H)$, and we
obtain compatible isomorphisms $G/Z(G)\cong H/Z(H)$ (both are
elementary abelian of order $p^{2n}$) and $[G,G]\cong [H,H]$ (both
are cyclic of order $p$), so $G$ and $H$ are isoclinic. The
assertion of the corollary now follows from Proposition
\ref{prop:extra}.
\end{proof}

\begin{prop} \label{prop:Black}
If $G$ is a Blackburn group, then $B_0(G)=0$.

\end{prop}

First recall the needed definitions.

\begin{defn} \label{def:Ded}
A group $G$ is called a {\it Dedekind group} if any subgroup of
$G$ is normal \cite[p. 33]{Be}.
\end{defn}

\begin{remark} \label{rem:Ded}
All Dedekind groups are classified \cite[pp.~33--34]{Be}. A Dedekind
group is either abelian, or a direct product of a quaternion group
of order 8 and an abelian group without elements of order 4. In both
cases we have $B_0(G)=0$.
\end{remark}

\begin{defn}
A non-Dedekind group $G$ is called a {\it Blackburn group}
if the intersection of all its non-normal subgroups is nontrivial.
\end{defn}

All such groups are classified  \cite{Bl}, and in the proof below we
proceed case by case.

\begin{proof}[Proof of Proposition $\ref{prop:Black}$]
If $G$ is a $p$-group, then, according to \cite[Theorem~1]{Bl},
$p=2$ and $G$ is either a direct product of quaternion groups and
abelian groups, or contains an abelian subgroup of index 2. In both
cases, $B_0(G)=0$, taking into account items (iii), (iv) above and
the formula $B_0(G_1\times G_2)= B_0(G_1)\times B_0(G_2)$
\cite{Ka3}.

So suppose that $G$ is not a $p$-group. By \cite[Theorem~2]{Bl},
there are five types of such groups. For types (a), (b), (d) and (e)
the assertion is an immediate consequence of earlier considerations.
Indeed, groups of types (a) and (d) are abelian-by-cyclic, and we
use item (v) above. In case (b), $G$ is a direct product of abelian
and quaternion groups, and the argument of the preceding paragraph
works.  Groups of type (e) are direct products of quaternion,
abelian, and abelian-by-cyclic groups, and we proceed as above.

It remains to consider case (c), where $G$ contains a subgroup $H$
of index 2 with the following property: $H$ has an index two abelian
subgroup $A$ of exponent $2^nk$, $k$ odd. Let $S_p$ denote a Sylow
$p$-subgroup of $G$. If $p$ is odd, then $S_p$ is abelian, hence
$B_0(S_p)=0$. Consider $S=S_2$. If $S$ is a Dedekind group, then
$B_0(S)=0$ in light of the remark after Definition \ref{def:Ded}. If
$S$ is not a Dedekind group, then the intersection of its non-normal
subgroups is nontrivial because each non-normal subgroup of $S$ is a
non-normal subgroup of $G$ and $G$ is a Blackburn group. So $S$ is a
Blackburn group too, and $B_0(S)=0$ (see the first paragraph of the
proof). Thus $B_0(S_p)=0$ for all $p$, and therefore $B_0(G)=0$
 \cite{Bo}, \cite{BMP} (see the first paragraph of the section).
\end{proof}

Theorem \ref{th:main} now follows from Propositions \ref{prop:extra}
and \ref{prop:Black}.

\section{Concluding remarks} \label{sec:concl}

We collect here several general remarks and open questions.

\begin{question}
Let $G$ be a group belonging to class (x) or (xi) of Proposition $\ref{prop:rigid}$.
Is it true that $B_0(G)=0$?
\end{question}

Here is a more general question:

\begin{question}
Let $G$ be a $\Sha$-rigid group.
Is it true that $B_0(G)=0$?
\end{question}

Note that there are groups $G$ with $B_0(G)=0$ that are not $\Sha$-rigid.
Say, so are first counter-examples to $\Sha$-rigidity constructed by Burnside
\cite{Bu2}: these are groups of order 32 for which it is known that
$B_0(G)=0$ \cite{CHKP}.

Returning to the list of Proposition \ref{prop:rigid} and looking at the last item,
we may ask the following parallel questions:

\begin{question}
\item[(i)]
Let $G=G_1*G_2$ be a central product of groups such that
$B_0(G_1)=B_0(G_2)=0$. Is it true that $B_0(G)=0$?
\item[(ii)]
Let $G=G_1*G_2$ be a central product of groups such that the
corresponding generically free linear quotients $V_1/G_1$ and
$V_2/G_2$ are stably rational. Is it true that so is $V/G$?
\end{question}

Definitely, it is much more tempting to understand whether there exists
some intinsic relationship between $\Sha$-rigidity and Bogomolov multiplier
behind the empirical observations presented in this paper. The interested reader is referred
to \cite{Ku2} for some speculations around these eventual ties.

\bigskip

\noindent {\it Acknowledgements}. The first author was supported in
part by the National Center for Theoretic Sciences (Taipei Office).
The second author was supported in part by the Minerva Foundation
through the Emmy Noether Research Institute of Mathematics and by
the Israel Science Foundation, grant 1207/12; this paper was mainly
written during his visit to the NCTS (Taipei) in 2012. Support of
these institutions is gratefully acknowledged.


\renewcommand{\refname}{\centering{References}}


\begin{thebibliography}{CHKP}




\bibitem[BVY]{BVY}
V. Bardakov, A. Vesnin, M. K. Yadav, Class preserving automorphisms
of unitriangular groups, Internat. J. Algebra Computation 22 (2012),
no.~3, 12500233, 17~pp.

\bibitem[Be]{Be}
Ya. Berkovich, Groups of prime power order, Vol.~1, Walter de
Gruyter, Berlin, 2008.

\bibitem[Bl]{Bl}
N. Blackburn,
Finite groups in which the nonnormal subgroups have nontrivial intersection,
J. Algebra 3 (1966), 30--37.

\bibitem[Bo]{Bo}
F. A. Bogomolov, The Brauer group of quotient spaces by linear group
actions, Izv. Akad. Nauk SSSR Ser. Mat. 51 (1987), 485--516; English
transl. in Math. USSR Izv. 30 (1988), 455--485.

\bibitem[BB]{BB}
F. A. Bogomolov, C. B\"ohning, Isoclinism and stable cohomology of
wreath products, \url{arXiv:1204.4747}.

\bibitem[BMP]{BMP}
F. A. Bogomolov, J. Maciel, T. Petrov, Unramified Brauer groups of
finite simple groups of Lie type $A_{\ell}$, Amer. J. Math. 126
(2004), 935--949.

\bibitem[Bu1]{Bu1}
W. Burnside, Theory of groups of finite order, 2nd ed., Cambridge Univ.
Press, Cambridge, 1911; reprinted by Dover Publications, New York, 1955.

\bibitem[Bu2]{Bu2}
W. Burnside, On the outer automorphisms of a group, Proc. London
Math. Soc. (2) 11 (1913), 40--42.

\bibitem[CT]{CT}
J. F. Carlson, J. Th\'evenaz, Torsion endo-trivial modules, Algebr.
Represent. Theory 3 (2000), 303--335.

\bibitem[CHKK]{CHKK}
H. Chu, S.-J. Hu, M. Kang, B. \`E. Kunyavskii,
 Noether's problem and the unramified Brauer group for groups of order $64$,
Int. Math. Res. Not. 2010, no.~12, 2329--2366.

\bibitem[CHKP]{CHKP}
H. Chu, S.-J. Hu, M. Kang, Yu. G. Prokhorov, Noether's problem for
groups of order 32, J. Algebra 320 (2008), 3022--3035.

\bibitem[CK]{CK}
H. Chu, M. Kang, Rationality of $p$-group actions,
J. Algebra 237 (2001), 673--690.


\bibitem[FS]{FS}
W. Feit, G. M. Seitz, On finite rational groups and related topics,
Illinois J. Math. 33 (1989), 103--131.

\bibitem[FN]{FN}
M. Fuma, Y. Ninomiya, ``Hasse principle'' for finite $p$-groups with
cyclic subgroups of index $p\sp 2$, Math. J. Okayama Univ. 46
(2004), 31--38.

\bibitem[Go]{Go}
D. Gorenstein, Finite groups, Harper and Row, N. Y., 1968.

\bibitem[HL]{HL}
A. Herman, Y. Li, Class preserving automorphisms of Blackburn
groups, J. Australian Math. Soc. 80 (2006), 351--358.

\bibitem[He1]{He1}
M. Hertweck, Class-preserving automorphisms of finite groups, J.
Algebra 241 (2001), 1--26.

\bibitem[He2]{He2}
M. Hertweck, Contributions to the integral representation theory of
groups, Habilitationsschrift, Univ. Stuttgart, 2004, available
online at
\url{http://elib.unistuttgart.de/opus/volltexte/2004/1638}.

\bibitem[HJ]{HJ}
M. Hertweck, E. Jespers, Class-preserving automorphisms and the
normalizer property for Blackburn groups, J. Group Theory  12
(2009),  157--169.

\bibitem[HKK]{HKK}
A. Hoshi, M. Kang, B. E. Kunyavskii, Noether's problem and
unramified Brauer groups, to appear in Asian J. Math.

\bibitem[HK]{HK}
S.-J. Hu, M. Kang, Noether's problem for some $p$-groups, in:
``Cohomological and geometric approaches to rationality problems'',
Progr. Math., vol.~282, Birkh\"auser, Boston, MA, 2010,
pp.~149--162.


\bibitem[Ka1]{Ka1}
M. Kang, Noether's problem for $p$-groups with a cyclic subgroup of index $p^2$,
Adv. Math. 226 (2011), 218--234.

\bibitem[Ka2]{Ka2}
M. Kang, Retract rational fields, J. Algebra 349 (2012), 22--37.

\bibitem[Ka3]{Ka3}
M. Kang, Bogomolov multipliers and retract rationality for semi-direct products,
\url{arXiv:1207.5867}.

\bibitem[KV1]{KV1}
M. Kumar, L. R. Vermani, ``Hasse principle'' for
extraspecial $p$-groups, Proc. Japan Acad. Ser. A Math. Sci. 76
(2000), 123--125.

\bibitem[KV2]{KV2}
M. Kumar, L. R. Vermani, ``Hasse principle'' for groups of
order $p\sp 4$, Proc. Japan Acad. Ser. A Math. Sci. 77 (2001), 95--98.

\bibitem[KV3]{KV3}
M. Kumar, L. R. Vermani, On automorphisms of some
$p$-groups, Proc. Japan Acad. Ser. A Math. Sci. 78 (2002), 46--50.

\bibitem[Ku1]{Ku1}
B. E. Kunyavski\u\i , The Bogomolov multiplier of finite simple
groups, in: ``Cohomological and geometric approaches to rationality
problems'', Progr. Math., vol.~282, Birkh\"auser, Boston, MA, 2010,
pp.~209--217.

\bibitem[Ku2]{Ku2}
B. Kunyavski\u\i , Local-global invariants of finite and infinite
groups: around Burnside from another side, preprint, 2012,
submitted.


\bibitem[La]{La}
R. Laue, On outer automorphism groups, Math. Z. 148 (1976),
177--188.

\bibitem[Mo1]{Mo1}
P. Moravec, Groups of order $p^5$ and their unramified Brauer
groups, J. Algebra 372 (2012), 320--327.

\bibitem[Mo2]{Mo2}
P. Moravec, Unramified Brauer groups and isoclinism, \url{
arXiv:1203.2422}.

\bibitem[On]{On}
T. Ono, A note on Shafarevich--Tate sets for finite groups, Proc.
Japan Acad. Ser. A Math. Sci. 74 (1998), 77--79.





\bibitem[OW]{OW}
T. Ono, H. Wada, ``Hasse principle'' for symmetric and
alternating groups,  Proc. Japan Acad. Ser. A Math. Sci. 75 (1999),
61--62.

\bibitem[Pe]{Pe}
E. Peyre, Unramified cohomology of degree $3$ and Noether's problem,
Invent. Math. 171 (2008), 191--225.



\bibitem[Sa1]{Sa}
D. J. Saltman, Noether's problem over an algebraically closed field,
Invent. Math. 77 (1984), 71--84.

\bibitem[Sa2]{Sa2}
D. J. Saltman, The Brauer group and the center of generic matrices,
J. Algebra 97 (1985), 53--67.

\bibitem[Su]{Su}
M. Suzuki, On finite groups with cyclic Sylow subgroups for all
odd primes, Amer. J. Math. 77 (1955), 657--691.

\bibitem[Wa]{Wa}
C. T. C. Wall, On the structure of finite groups with periodic
cohomology, preprint, 2010, available at
\url{www.liv.ac.uk/~ctcw/FGPCnew.pdf}.




\bibitem[Ya1]{Ya1}
M. K. Yadav, On automorphisms of some finite $p$-groups, Proc.
Indian Acad. Sci. (Math. Sci.) 118 (2008), 1--11.

\bibitem[Ya2]{Ya2}
M. K. Yadav,
Class preserving automorphisms of finite $p$-groups: A survey,
in: ``Groups St. Andrews 2009 in Bath., vol.~2'', London Math. Soc.
Lecture Note Ser., vol.~388,  Cambridge Univ. Press, Cambridge {\it et al.},
2011, pp.~569--579.


\end{thebibliography}
\end{document}